\documentclass{article}[12pt] 
\usepackage{amssymb}
\usepackage{graphicx}
\usepackage{natbib}
\newtheorem{thm}{Theorem}
\newtheorem{lem}{Lemma}
\newtheorem{defn}{Definition}

\begin{document}
\noindent
UDK 368:519.21

\noindent
B.V. NORKIN 

\bigskip
\noindent
SAMPLE AVERAGE APPRIXIMATION OF MULTIOBJECTIVE 
\\STOCHASTIC OPTIMIZATION PROBLEMS 
\\WITH APPLICATION IN INSURANCE
%

\begin{abstract}
The article describes a technique for solving multiobjective stochastic optimization problems. As a generalized model of a stochastic system to be optimized a vector   "input-random output" system is used. Random outputs are converted into a vector of deterministic performance and risk indicators. The problem is to find those inputs that correspond to a Pareto-optimal values of output indicators. The problem is approximated by a sequence of deterministic multi-criteria optimization problems, where, for example, the objective vector
function is a sample average approximation of the original one and the feasible set is a discrete sample approximation of the feasible inputs. Approximate
optimal solutions are defined as weakly Pareto-efficient ones within some vector
tolerance. Convergence analysis includes establishing convergence of the general
approximation scheme and establishing conditions of convergence with probability one under proper regulation of sampling parameters. The proposed technique is illustrated on a computer system for supporting multi-criteria optimization of insurance business.
\end{abstract}



\section{Introduction}
Contemporary approach to optimal decision making is based on system modeling and system optimization. Any complex system can be considered as an "Input-Output" system $y=g(x)$, where $x$ denotes the input parameter vector from some feasible set ${\rm X}$, and $y$ is an output parameter vector from a set ${\rm Y}$, $g$ is some mapping of ${\rm X}$ into ${\rm Y}$. The optimization (e.g., maximization) is applied to some utility functional $ f (x, y) $ under constraints $ y = g (x) $, $ x \in {\rm X}$. The vector of input parameters can often be divided into controlled part, which still is denoted by $ x$, and uncontrolled one $ \omega$ that takes values from some set $ \Omega$. Thus, the model becomes $ y = g(x,\omega)$. The model is usually given by a simulation computer program or as an output result of an optimization solver.

  The vector $ \omega $ of uncertain parameters can be either deterministic or random with distribution $ P$. In the first case, the optimization problem reads as: $ \min _ {\omega \in \Omega} f (x, y = g (x,\omega)) \to \max _ {x \in X}$ that corresponds to the so-called minimax decision-making approach. In the second case, the problem relates to the stochastic programming and, in particular, can be formulated as: $ F (x) = {\rm E} f (x, y = g (x, \omega)) \to \max _ {x \in X}$, where $ {\rm E}$ denotes the mathematical expectation operation \cite{Ermoliev_1976rus}, \cite{Shapiro_Dentcheva_Ruszczynski_2009}. 
	
	However, efficient and unambiguous choice of a utility function $ f $ is not always possible. Although a preference relation can exist on the set of "input-output" pairs $ {\rm X \times Y}$ or on ${\rm Y}$  that allows to consider only non-dominated "input-output" pairs. This situation relates to the deterministic multi-criteria optimization. If the model $ y = g (x, \omega) $ contains uncertain parameters $ \omega $, then this situation is qualified as an uncertain programming and can be formalized in different ways \cite{Liu_2005}. Note that the stochastic programming problem already contains a vector criterion $ f (x, y = g (x, \omega)) $, $ \omega \in \Omega$, with a large number of components that, in any case, are reduced to one or more indicators. The most commonly used indicator is the average value of $ F (x) = {\rm E} f (x, y = g (x, \omega))$, along with the variance functions, probability, quantile (VaR) and other risk indicators \cite{Shapiro_Dentcheva_Ruszczynski_2009}, \cite{Gutjahr_Pichler_2013},\cite{Stancu-Minasian_1984}.  Optimization of such indicators requires substantial computational resources, and in particular, use of parallel computing. In this paper we consider the usage of parallel calculations for solving multiobjective stochastic optimization problems of actuarial mathematics.

\section {Literature review.} 

Unlike standard one-criterion stochastic programming problems \cite{Ermoliev_1976rus},
\cite{Shapiro_Dentcheva_Ruszczynski_2009} the problem of output vector optimization can contain non-convex,
non-smooth or even discontinuous functions, so traditional stochastic programming methods such as gradient type procedures \cite{Ermoliev_1976rus} with a random starting point might not be applicable. In this case random search methods, for example, evolutionary or hybrid algorithms should be applied  \cite{Branke_et_al_2008}, \cite{Coello_2006}, \cite{Gutjahr_2012}, 
\cite{Konak_Coit_Smith_2006}, \cite{Marler_Arora_2004}. In case of small dimension $ n $ of the set $ X \subset {\rm R}^{n}$ a simple method of uniform random search can appear to be competitive, moreover it allows natural parallelization. This method randomly generates a cloud of points in the feasible region, and then non-dominated points are selected. In the vicinity of the non-dominated points a new random points are generated, again non-dominated points are chosen and so on. Efficiency is boosted due to the fact that the new points are generated in the perspective areas. However, in solving multidimensional combinatorial problems the simple random search method gives the way for more sophisticated evolutionary algorithms, see e.g. \cite{Kleeman_Lamont_2008}, \cite{Zilinskas_2014}. We propose an interactive multi-criteria method of controlled random search, where the region for generation of new random points is defined by the decision making person on the basis of visual analysis of the set of non-dominated points from the previous generation. In this paper the results of previous author's works \cite{Norkin_2011_HPC-UA}, \cite {Norkin_2012_KPIeng}, \cite{Norkin_2014_CSA2} on simulation of insurance business are extended towards the dynamic stochastic multi-criteria optimization with parallel computations. Performance indicators of an insurance company are estimated by means of parallel Monte Carlo method, and Pareto-optimal (or non-$\epsilon $-dominated) set is constructed by means of interactive random search method.  In the present paper we focus on implementation and convergence analysis of the method. Other non-scalarization approaches to stochastic multi-criterion optimization are reviewed in \cite{Gutjahr_Pichler_2013}. Overview of computer systems that implement interactive deterministic multi-criteria optimization is available in \cite{Poles_Vassileva_Sasaki_2008}. Usage of parallel calculations in multi-criteria optimization is discussed in \cite{Tabli_et_al_2008}.  In \cite{Fliege_Xu_2011} another (scalarization) approach to vector stochastic optimization was studied: the vector performance indicators such as mathematical expectations are evaluated by means of Monte Carlo method and then are aggregated into a scalar performance indicators by means of non-linear utility functions. In \cite{Jevne_Haddow_Gaivornski_2012} evolutionary multiobjective optimization algorithms are applied to search for Pareto-optimal portfolios according to the criteria of the average return and high-level quintile.

\section {Stochastic multiobjective optimization} 

Paper \cite{Gutjahr_Pichler_2013} provides a contemporary review of stochastic multiobjective optimization problem settings. In context of "input-output" models
when the utility function is unknown, we have to deal directly with the vector model $ y = g (x) $, which maps the set of inputs $ X $ into a set of outputs $ Y = f (X) = \left \{y = f (x), \, x \in X \right \} \subseteq {\rm Y}$. At the same time some kind of preference relation $ \prec$ is available on the set of outputs. In case of $ Y \subseteq {\rm R} ^ {m} $ this relation $ \prec $ is usually defined by a cone of nonnegative vectors $ {\rm R}_{+}^{m} = \left \{y = (y_ {1}, ..., y_ {m}) \in {\rm R}^{m}: \, \, y_{i} \ge 0, \, 1 \le i \le m \right \} $. In this case, $ y_{1} \prec y_ {2}$ if and only if $ \left (y_{2} -y_{1} \right) \in {\rm R}_{+}^{m}$. The set of $\prec$-optimal vectors $Y^{*}$ is a subset of $ Y = f (X) $ such that there is no $y\in Y$, $y^*\prec y$ and $ y \ne y^{*}$. Corresponding set of inverse images of $ X^{*} = \left \{x \in X: \, \, f(x) \in Y^{*} \right \}$ is called the set $ \prec$-optimal solutions. Vector optimization problem consists in  finding or approximating the sets $ Y ^ {*} $ and $ X ^ {*} $. 
If the model contains uncertain or stochastic parameters $ \omega $ and instead of a utility function only a preference relation $ \prec$ is defined on the set of outputs, then we are dealing with the problem of a vector (or multi-criteria) stochastic programming \cite{Gutjahr_Pichler_2013},\cite{Stancu-Minasian_1984}. In this case for each $ \omega$, there is a multiple-output $ Y _ {\omega} = f _ {\omega} (X) = \left \{y = f (x, \omega), \, x \in X \right \}$ and its $ \prec$-optimal set $ Y _ {\omega} ^ {*} $ and we have to determine what set can be considered as a common optimal set for all $ \omega$. To do this, we have to define additional preference relation on the set of random vectors $f(x,\omega)$.

In case of one-dimensional stochastic model, when $ y = f (x, \omega) \in {\rm R}$, the required preference relation can be a stochastic dominance of the first $ \prec_{(1)} $ or higher orders $ \prec_{(i)}$, $i\ge 2$, \cite{Muller_Stoyan_2002}. 
Let us remind that $ f (x_ {1}, \omega) \prec _ {(1)} f (x_ {2}, \omega) $, if for distribution functions $ F_ {1} (t) $ and $ F_ {2} (t) $ of random variables $ f (x_ {1}, \omega) $ and $ f (x_ {2}, \omega) $, it fulfills $ F_ {1} (t) \ge F_ {2} (t) $ for all $ t \in {\rm R}$. 
First order relation $\prec_{(1)}$ is verifiable in case of a finite discrete distribution of $\omega$.
The second order relation $ f (x_ {1}, \omega)\prec_{(2)}f(x_{2},\omega)$ holds if ${\rm E} u\left(f(x_{1},\omega)\right)\ge{\rm E} u\left (f(x_{2},\omega)\right)$ for all concave non-decreasing functions $u(\cdot)$ such that mathematical expectations exist. This definition can be extended to vector random variables, but verification of the stochastic dominance relation appears to be rather difficult as it involves checking an infinite number of inequalities.

We can consider the preference relation $ \prec _ {{\rm E}} $ on the basis of expectations, $ f (x_ {1}, \omega) \prec _ {{\rm E}} f (x_ {2}, \omega) $ if and only if $ {\rm E} f (x_ {1}, \omega) \le {\rm E} f (x_ {2}, \omega) $ (component-wise) and $ {\rm E} f (x_ {1}, \omega) \ne {\rm E} f (x_ {2}, \omega) $. As the optimal set one can take $ \prec_{{\rm E}}$-optimal subset of $ \left \{{\rm E} f (x, \omega), \, \, x \in X \right \} $. It is known that the set of $ \prec _ {{\rm E}} $-optimal points in the convex case can be obtained by optimizing the set of linear $ \left \langle w, {\rm E} f (x, \omega) \right \rangle $, $ w \in W $, or nonlinear $ U ({\rm E}f (x, \omega)) $, $ U \in {\rm U} $ scalar convolutions of vector criterion $ {\rm E} f (x, \omega) $ \cite {Jahn_2011}. In \cite {Fliege_Xu_2011} expectations $ {\rm E} f (x, \omega) $  were replaced by their empirical estimates $ {\rm E} _ {N} f (x, \omega) $ and the convergence of the set of optimal points $ U \left ({\rm E} _ {N} f (x, \omega) \right) \to extr_ {x \in X} $, $ U \in {\rm U} $, to $ \prec _ {{\rm E}} $-optimal set in case of $ N \rightarrow \infty $ was studied. 

In a similar way preference relations $ \prec_{P} $ and $\prec_{Q} $ are introduced, based on the comparison of sets of probabilities $ \Pr \left \{f_ {i} (x, \omega) \ge t_ {i} \right \} $, $ i = 1, ..., m, $ and quantiles $ Q_ {i} (x, q_ {i }) = \inf \left \{t: \, \, \Pr \left \{f_ {i} (x, \omega) \ge t \right \} \le q \right \} $, $ i = 1, ..., m $, respectively. 

Along with the average values ${\rm E} f (x, \omega) = \left ({\rm E} f_ {1} (x, \omega), ..., {\rm E} f_ {m} (x, \omega) \right) $ it is worthwhile to consider the standard deviations $ \sigma _ {i} (x) = \left ({\rm E} (f_ {i} (x, \omega) - {\rm E} f_ {i} (x, \omega))^{2} \right)^{1/2} $ or standard semi-deviations $ \sigma _ {i}^{+} (x) = \left ({\rm E } \, {\rm max}^{2} \left \{0, f_ {i} (x, \omega) - {\rm E} f_ {i} (x, \omega) \right \} \right )^{1/2} $ of indicators $ f_ {i} (x, \omega) $ from their average values $ {\rm E} f_ {i} (x, \omega) $, $ i = 1, .. ., m $, as well as vector indicators $ \left ({\rm E} f_ {1} (x, \omega) + \alpha _ {1} \sigma _ {1} (x), ..., {\rm E} f_ {m} (x, \omega) + \alpha _ {m} \sigma _ {m} (x) \right) $. In \cite {Caballero_et_al_2001} the relationships between Pareto optimal sets of various deterministic formalizations of the multiobjective stochastic  optimization problem were studied, including the problems: 

$ {\rm E} f (x, \omega) \to \min _ {x \in X} $, 

$ \sigma (x) = \left (\sigma _ {1} (x), ..., \sigma _ {m} (x) \right) \to \min _ {x \in X } $, 

$ \left ({\rm E} f_ {1} (x, \omega) + \alpha _ {1} \sigma _ {1} (x), ..., {\rm E} f_ {m} (x, \omega) + \alpha _ {m} \sigma _ {m} (x) \right) \to \min _ {x \in X} \\ 
\left(\alpha = (\alpha _ {1}, ..., \alpha _ {m}) \in {\rm R} _ {+}^{m} \right) $, 

$ \left \{\Pr \left \{f_ {i} (x, \omega) \ge t_ {i} \right \}, \, \, i = 1, ..., m \right \} \to \max _ {x \in X} $, 

$ \left \{Q_ {i} (x, q_ {i}), \, \, i = 1, ..., m \right \} \to \min _ {x \in X} $. 

The following concept appears to be useful for control of accuracy, strength and direction of dominance.

\begin{defn}
($ \vec {\epsilon} $-dominance and $ \vec {\epsilon} $-efficiency/optimality). Vector $ \vec {g} _1 \in {\rm R}^m $ $\; \vec {\epsilon} $-dominates the vector $ \vec {g} _2 \in {\rm R}^m $, if $ \vec {g} _1> \vec {g} _2 + \vec {\epsilon} $ (component-wise), where $ \vec {\epsilon} \in {\rm R}^m $. Subset of $ G^* _ {\vec {\epsilon}} $ of the set $ G \subset {\rm R}^m $ is called $ \vec {\epsilon} $-efficient/optimal if for any $ \vec {g } \in G^* _ {\vec {\epsilon}} $ there is no $ \vec {g} \, '\in G $, $ \vec {g} \,' \neq \vec {g} $, such that $ \vec {g} \, '> \vec {g} + \vec {\epsilon} $. 
\end{defn}

The concept of $ \vec {\epsilon} $-efficiency was introduced in \cite {Kutateladze_1979}. In case of $ \vec {\epsilon}> 0 $, it generalizes the standard notion of $ {\epsilon} $-optimality of scalar optimization. In particular the concept of $ \vec {\epsilon} $-efficiency includes the notion of weak Pareto optimality \cite{Podinovski_Nogin_1982} 
that corresponds to $ \vec {\epsilon} = 0 $. Further various generalization of the  $ \vec {\epsilon} $-efficiency concept are discussed in \cite{Gutjahr_2012}, \cite{Gutierrez_Jimenez_Novo_2006}, \cite {Gutierrez_Jimenez_Novo_2012}. By adding $ \vec {\epsilon} $ to a vector $ \vec {g} $ the importance of components of $ \vec {g} $ can be controlled (the importance of the criteria in vector optimization), namely, increase of $ \epsilon_i $ component decreases the importance of $ g_i $ component. Moreover, in contrast to \cite {Kutateladze_1979}, \cite {Gutierrez_Jimenez_Novo_2006}  we allow $ \vec {\epsilon} \notin {\rm R}^m _ + $. If $ \vec {\epsilon} $ contains negative components, then $ \vec {\epsilon} $-dominance of $ \vec {g} _1 $ over $ \vec {g} _2 $ admits that some components of $ \vec {g} _1 $ can be somewhat smaller than the corresponding components of $ \vec {g} _2 $. Thus we note that if the point $ \vec {g}^* \in G $ is $ \vec {\epsilon} _k $-efficient for some sequence $ \{{\rm R}^m \ni \vec {\epsilon} _k \rightarrow 0, \, k = 1,2, ... \} $, then $ \vec {g}^* $ is called a generalized efficient point (see. \cite [Definition 5.53] {Mordukhovich_2006}).  

Let us remind some notation and definitions \cite[Sec. 4A]{Rockafellar_Wets_1998}, that concern the convergence of a sequence of sets $\{Z_i\subset {\rm R}^n,\,i=1,2,...\}$: $\limsup_i Z_i=\{z: \exists\; z_{i_k}\in Z_{i_k}, \; z=\lim_k z_{i_k}\}$, $\liminf_i Z_i=\{z: \exists\; z_{i}\in Z_{i}, \; z=\lim_i z_{i}\}$, $\lim_i Z_i=\liminf_i Z_i=\limsup_i Z_i$. 

\begin{lem}\label{Lemma_1}
(Properties of the $\vec{\epsilon}$-optimal mappings). Let sequence of sets $\{Z_i\in{\rm R}^m\}$ converges to a compact set $\{Z\subset{\rm R}^m\}$, $\lim_i Z_i=Z$. Denote $Z^*_i(\vec{\epsilon})$, $Z^*(\vec{\epsilon})$ subsets of $\vec{\epsilon}$-nondominated points in $Z_i$ and $Z$, respectively. Let $\lim_i\vec{\epsilon}_i=\vec{\epsilon}$. Then for any $\vec{\epsilon}\,'\le\vec{\epsilon}$, $\vec{\epsilon}\,'\ne\vec{\epsilon}$, it holds true 
\[ Z^*(\vec{\epsilon}\,')\subseteq\liminf_i Z^*_i(\vec{\epsilon}_i)\subseteq\limsup_i Z^*_i(\vec{\epsilon}_i)\subseteq Z^*(\vec{\epsilon}), \] 
where the last inclusion, in particular, indicates that the mapping $\vec{\epsilon}\rightarrow Z^*(\vec{\epsilon}) $ is upper semicontinuous. More over, in case of convex set $Z$ we have
$\liminf_{i}Z^{*}_i(\vec{\epsilon}_i)=
\limsup_{i}Z^{*}_i(\vec{\epsilon})=
Z^{*}(\vec{\epsilon}).$
\end{lem}

{\it Proof.}
Let us prove the first inclusion $Z^*(\vec{\epsilon}\,')\subseteq \liminf_i Z^*_i(\vec{\epsilon}_i)$, i.e. that for each point $\vec{z}^{\,*}\in Z^*(\vec{\epsilon}\,')$ there exists a sequence $\{Z^*_i(\vec{\epsilon}_i)\ni \vec{z}^{\,*}_i\rightarrow \vec{z}^{\,*}\}$. 
Let us fix $\vec{z}^{\,*}\in Z^*(\vec{\epsilon}\,')$ 
and assume that $\vec{z}^{\,*}\notin \liminf_i Z^*_i(\vec{\epsilon}_i)$. Since $\vec{z}^{\,*}\in Z^*(\vec{\epsilon}\,')\subseteq Z=\lim_i Z_i$, there exists a sequence $Z_i\ni \vec{z}_i\rightarrow \vec{z}^{\,*}$. Due to assumption $\vec{z}^{\,*}\notin \liminf_i Z^*_i(\vec{\epsilon}_i)$ there is an infinite subsequence $\{\vec{z}_{i_k}\notin Z^*_{i_k}(\vec{\epsilon}_{i_k})\}$. So there are $\vec{z}\,'_{i_k}\in Z_{i_k}(\vec{\epsilon}_{i_k})$ such that $\vec{z}\,'_{i_k}\ge\vec{z}_{i_k}+\vec{\epsilon}_{i_k}$. Since $Z$ is a compact set then, without loss of generality, we can assume that $\vec{z}\,'_{i_k}\rightarrow \vec{z}\,'\in Z$ 
and hence $\vec{z}\,'=\lim_k \vec{z}\,'_{i_k}\ge \lim_k \vec{z}_{i_k}+\lim_k\vec{\epsilon}_{i_k}=\vec{z}^{\,*}+\vec{\epsilon}\ge\vec{z}^{\,*}+\vec{\epsilon}\,'$ and
$\vec{z}\,'\ne\vec{z}^{\,*}+\vec{\epsilon}\,'$.  
This means that point $\vec{z}^{\,*}$ is $\vec{\epsilon}\,'$-dominated, a contradiction. 

Let us prove the second inclusion, $\limsup_i Z^*_i(\vec{\epsilon}_i)\subseteq Z^*(\vec{\epsilon})$. Assume the contrary, that there is a subsequence $\{Z^*_{i_k}(\vec{\epsilon}_{i_k})\ni \vec{z}^{\,*}_{i_k}\rightarrow \vec{z}\,'\notin Z^*(\vec{\epsilon})\}$. Then there is a point $\vec{z}\,''\in Z$ and a sequence  $\{\vec{z}_i\in Z_i\}$ such that $\vec{z}\,''\ge\vec{z}\,'+\vec{\epsilon}$,
$\vec{z}\,''\ne\vec{z}\,'+\vec{\epsilon}$ and $ \vec{z}_i\rightarrow \vec{z}\,''$. So we obtain $ \lim_{k}\vec{z}_{i_k}=\vec{z}\,''\ge\vec{z}\,'+\vec{\epsilon}=\lim_k\left(\vec{z}^{\,*}_{i_k}+\vec{\epsilon}_{i_k}\right),$ and
$ \lim_{k} \vec{z}_{i_k}=\vec{z}\,''\ne\vec{z}\,'+\vec{\epsilon}=\lim_k \left(\vec{z}^{\,*}_{i_k}+\vec{\epsilon}_{i_k}\right),$
 that contradicts $\vec{\epsilon}_{i_k}$-efficiency of points $\vec{z}^{\,*}_{i_k}$ for sufficiently large $k$. The proof is complete. 

In the context of stochastic multi-criteria optimization, in the present paper we deal with the following vector optimization problem: \[\left \{{\rm E} f_ {i} (x, \omega), \, \, \sigma _ {i}^{+} (x), \, \, P \left \{f_ {i} (x, \omega) \le u_ {i} \right \}, \, \, i = 1, ..., m \right \} \to \min _ {x \in X}, \] for which we build approximations of $ \vec {\epsilon} $-optimal sets. Here performance indicators $ {\rm E} f_ {i} (x, \omega) $ serve as utility measures, but $ \sigma _ {i}^{+} (x) $ and $ P \left \{f_ {i} (x, \omega) \le u_ {i} \right \} $ represent risk measures for solution $ x $. 

\section{Approximation of multi-criteria optimization problems} 

Consider the general problem of multi-criteria optimization of the form \begin{equation}\label{MOP2} \vec{F}(x)=\{f_1(x),...,f_m(x)\}\rightarrow \max_{x\in X\subset{\rm R}^n}, \end{equation}where the functions $f_i(x),\;i=1,...,m,$ assumed to be continuous on a compact set $X\subset {\rm R}^n$, and a preference relation in the criteria space ${\rm R}^m$ are set by the nonnegative vector cone ${\rm R}^m_{++}=\{x\in {\rm R}^m:\,x_i>0,\,i=1,...,m\}$. The problem is about finding the of week Pareto-optimal set $X^*$ and a set $X^*(\vec{\epsilon})$ $\vec{\epsilon}$-effective points for(\ref{MOP2}). 

It is easy to see, that the mapping $\vec{\epsilon}\rightarrow X^*(\vec{\epsilon})$ is semi-continuous form above ie.  $\limsup_i X^*(\vec{\epsilon}_i)\subseteq X^*(\vec{\epsilon})$ for any subsequence $\vec{\epsilon}_i\rightarrow\vec{\epsilon}$. Really, let us assume the contrary Then for some sequence $\vec{\epsilon}_i\rightarrow\vec{\epsilon}$ exists $X^*(\vec{\epsilon}_{i_k})\ni x_{i_k}\rightarrow x'\notin X^*(\vec{\epsilon})$. Due to dominance $x'$ there is a point $x''\in X$ such that $\vec{F}(x'')>\vec{F}(x')+\vec{\epsilon}=\lim_k\vec{F}(x_{i_k})+\vec{\epsilon}$, that contradicts to the concept of dominance $x_{i_k}$. 

Let us consider also an approximations of the problem  (\ref{MOP2}): \begin{equation}\label{MOP2_approx} \vec{F}^i(x)=\{f_1^i(x),...,f_m^i(x)\}\rightarrow \max_{x\in X_i\subset{\rm R}^n}, \;\;\;i=1,2,..., \end{equation} where the sequence of sets $\{X_i\}$ converges to the set $X$, $\lim_i X_i=X$, and the sequence of vector functions $\{\vec{F}^i(x), \;x\in X_i\}$ converges to the vector function $\vec{F}(x),\;x\in X$ (in the sense of definitions \ref{Def2} adn \ref{Def3}). Denote $X^*_i(\vec{\epsilon})$ the set of $\vec{\epsilon}$ nondominated points of the problem (\ref{MOP2_approx}). 

\begin{defn} \label{Def2}
(continuous convergence of the sequence of vector functions) The sequence of vector functions $\vec{F}^i(x),\;x\in X_i$ is called continuously convergent to the vector function $\vec{F}(x),\;x\in X$, if a) $\lim_i X_i=X$, b) for any sequence $X_i\ni x_i\rightarrow x$ is true $\limsup_i\vec{F}^i(x_i)= \vec{F}(x)$ (component wise), 
\end{defn}

\begin{defn}\label{Def3} 
(graphic convergence of sequence of vector functions from below ). Sequence of vector functions $\vec{F}^i(x),\;x\in X_i $ is called graphically converging from below to the vector function $\vec{F}(x),\;x\in X$, if a) $\lim_i X_i=X$, b) for each sequence $X_i\ni x_i\rightarrow x$  we have $\limsup_i\vec{F}^i(x_i)\leqslant \vec{F}(x)$ (component-wise), and c) for any point $x\in X$ there is a sequence $X_i\ni x_i\rightarrow x$ such that $\lim_i\vec{F}^i(x_i)=\vec{F}(x)$. 
\end{defn}

The concepts of continuous and graphical convergence of multivalued mappings and functions (epi- and hypo-convergence) were studied in detail in \cite[Sec. 6E, 6G, 7B]{Rockafellar_Wets_1998}. Definitions \ref{Def2} and \ref{Def3} differ from the corresponding notions from \cite{Rockafellar_Wets_1998} by the fact that in definitions \ref{Def2} and \ref{Def3} domains $X_i,\, X$ of functions $\vec{F}^i,\, \vec{F}$ are clearly outlined , and in \cite[Def. 5.41]{Rockafellar_Wets_1998} functions are considered to be defined on the common domain $X$ or ${\rm R}^n$. And besides, the definition \ref{Def3} extends the definition of the graphic convergence of scalar functions \cite[7(3), 7(9), Def. 7.1]{Rockafellar_Wets_1998} to vector functions. 

{\bf Examples} 
of graphically convergent from below sequences of vector functions. 

{\bf E1.} Obviously, if a sequence $\{\vec{F}^i(x),\;x\in X_i\}$ converges continuously to $\{\vec{F}(x),\;x\in X\}$, i.e. $\lim_i X_i=X$ and $\lim_i\vec{F}^i(x_i)=\vec{F}(x)$ for any sequence $X_i\ni x_i\rightarrow x$, then $\{\vec{F}^i(\cdot)\}$ converges to $\vec{F}(\cdot)$ graphically. 

{\bf E2.} Obviously, if all, except the first, scalar components of $\{\vec{F}^i(\cdot)\}$ converge continuously to the corresponding scalar components of $\{\vec{F}(\cdot)\}$, and the first component of $\{\vec{F}^i(\cdot)\}$ hypo-converges to the first component of $\vec{F}(\cdot)$, then $\{\vec{F}^i(\cdot)\}$ graphically converges  from below to $\vec{F}(\cdot)$. 

{\bf E3.} If $\vec{F}^i(x)=\vec{F}(x,y_i),\;x\in X,\;Y\ni y_i\rightarrow y$, where $\vec{F}(x,y)$  is component-wise upper semicontinuous on $X\times Y$ and is continuous at $y\in Y$ for any $x\in X$, then $\{\vec{F}^i(\cdot,y_i)\}$ graphically  converges  from below to $\vec{F}(\cdot,y)$. In particular, this case includes a stationary sequence of upper semicontinuous vector functions $\vec{F}^i(x)=\vec{F}(x),\;x\in X_i=X$. 

{\bf E4.} Example of construction of a continuously convergent sequence of functions. Let function $\vec{F}(x),\;x\in X$  be continuous on a closed set $X$, and a sequence of functions $\{\vec{F}^i(x),\;x\in X_i\subseteq X\}$ is such that $\Delta_i:=\sup_{x\in X_i}\|\vec{F}^i(x)-\vec{F}(x)\|\rightarrow 0$ with $i\rightarrow\infty$. Then, for any sequence $(X_i\ni) x_i\rightarrow x$, we obtain \[ \|\vec{F}^i(x_i)-\vec{F}(x)\|\le \|\vec{F}^i(x_i)-\vec{F}(x_i)\|+\|\vec{F}(x_i)-\vec{F}(x)\| \le \Delta_i+ \|\vec{F}(x_i)-\vec{F}(x)\|\rightarrow 0. \] 

{\bf E5.} If the objective vector function of problem (\ref{MOP2}) has the form of expectation, 
$\vec{F}(x)={\rm E}\vec{F}(x,\omega)$, $x\in X$, then the sample average approximations 
$\vec{F}^i(x)=\left(1/M_i\right)\sum_{k=1}^{M_i}\vec{F}(x,\omega_k)$ can be used instead of 
$\vec{F}(x)$, where $\{\omega_k,\; k=1,2,...\}$ are i.i.d. observations of the random parameter 
$\omega$. Terms of uniform and, therefore, continuous convergence of empirical estimates 
$\vec{F}^i(x)$ to $\vec{F}(x)$ on the set $X$ can be found in 
\cite[Sec. 7.2.5]{Shapiro_Dentcheva_Ruszczynski_2009}. 

Next we present sufficient conditions of continuous convergence of discretely defined empirical functions $\vec{F}^i$ to a continuous expectation function $\vec{F}$. Assume that functions 
$\vec{F}(x,\omega)$ are uniformly bounded on $ X $, $\|\vec{F}(x,\omega)\|\le M$ and at each point $x\in X_i$ of a discrete set $X_i\subset X$ (with the number of elements $N_i$) an empirical estimate  $\vec{F}^i(x)=\left(1/M_i\right)\sum_{k=1}^{M_i}\vec{F}(x,\omega_k)$ is independently constructed such that 
\[ {\rm Pr}\left\{\|\vec{F}^i(x)-\vec{F}(x)\|>\delta\right\}\le 
C\exp\left\{-2M_i\delta^2/M^2\right\}\;\;\;\forall\;\delta>0. \] 
Such estimates follow, e.g.,  from Hoeffding's inequality 
\cite[Sec. 7.2.8]{Shapiro_Dentcheva_Ruszczynski_2009} with $C=2m$. 
Then $\Delta_i=\max_{x\in X_i}\|\vec{F}^i(x)-\vec{F}(x)\|\rightarrow 0$ with probability one, if, for example, $M_i\ge \alpha N_i$, $\alpha>0$, and the numerical sequence $\{N_i\}$ strictly monotonically increases to infinity. Indeed, the assertion follows from the fact that for any $\delta>0$ it holds true 
\[ \sum_{i=1}^\infty {\rm Pr}\left\{\Delta_i>\delta\right\}\le 
\sum_{i=1}^\infty CN_i\exp\left\{-2M_i\delta^2/M^2\right\}\le 
\sum_{i=1}^\infty CN_i\exp\left\{-2N_i\alpha\delta^2/M^2\right\}<+\infty. \] 

\begin{thm}\label{theo1} 
(Convergence of solutions of approximate problems (\ref{MOP2_approx})). Let a sequence of sets $\{X_i\}$ converges to a compact set  $X$, $\lim_i X_i=X$, and a sequence of functions $\{\vec{F}^i(x), \;x\in X_i\}$ graphically converges from below to a vector function  $\vec{F}(x),\;x\in X$. Let $\lim_i\vec{\epsilon}_i=\vec{\epsilon}$. Then for each $\vec{\epsilon}\,'<\vec{\epsilon}$ we have 
\[ X^*(\vec{\epsilon}\,')\subseteq\liminf_i X^*_i(\vec{\epsilon}_i)\subseteq \limsup_i X^*_i(\vec{\epsilon}_i)\subseteq X^*(\vec{\epsilon}). \] 
\end{thm}

{\it Proof.}
Let us show, that $\limsup_i X^*_i(\vec{\epsilon}_i)\subseteq X^*(\vec{\epsilon})$. Assume the contrary, that $X^*_{i_k}(\vec{\epsilon}_{i_k})\ni x^*_{i_k}\rightarrow x'\notin X^*(\vec{\epsilon})$. Since $x^*_{i_k}\in X^*_{i_k}\subseteq X_{i_k}$ and $\lim_i X_i=X$, then $x'\in X$. As $x'\notin X^*(\vec{\epsilon})$, then there exists $x''\in X$ such that $\vec{F}(x'')>\vec{F}(x')+\vec{\epsilon}\;$ and, due to graphical convergence from below of 
$\{\vec{F}^i(\cdot)\}$ to $\vec{F}(\cdot)$, it holds  $\vec{F}(x'')>\vec{F}(x')+\vec{\epsilon}\ge 
\lim_k\vec{F}^{i_k}(x^*_{i_k})+\vec{\epsilon}$. Since $\lim_i X_i=X$ then there is a sequence $\{x_i\in X_i\}$ such that $x_i\rightarrow x''$, $\vec{F}^i(x_i)\rightarrow \vec{F}(x'')$, and thus $x_{i_k}\rightarrow x''$, $\vec{F}^{i_k}(x_{i_k})\rightarrow \vec{F}(x'')$. Hence, $\lim_k\vec{F}^{i_k}(x_{i_k})=\vec{F}(x'')>\vec{F}(x')+\vec{\epsilon}\ge \lim_k\left(\vec{F}^{i_k}(x^*_{i_k})+\vec{\epsilon}_{i_k}\right)$ and for sufficiently large $k$ 
points $x_{i_k}$ $\vec{\epsilon}_{i_k}$-dominate points $x^*_{i_k}$. This contradiction proves the required assertion, $\limsup_i X^*_i(\vec{\epsilon}_i)\subseteq X^*(\vec{\epsilon})$. 

Let us now prove that $X^*(\vec{\epsilon}\,')\subseteq \liminf_i X^*_i(\vec{\epsilon}_i)$, i.e. for each point  $x^*\in X^*(\vec{\epsilon}\,')$ there exists a sequence $X^*_i(\vec{\epsilon}_i)\ni x^*_i\rightarrow x^*$. Let us fix $x^*\in X^*(\vec{\epsilon}\,')$ and suppose the contrary, that $x^*\notin \liminf_i X^*_i(\vec{\epsilon}_i)$. Since $x^*\in X^*(\vec{\epsilon}\,')\subseteq X$, then there is a sequence $X_i\ni x_i\rightarrow x^*$ such that $\lim_i\vec{F}^i(x_i)=\vec{F}(x^*)$. By the contrary assumption,  $x_{i_k}\notin X^*_{i_k}(\vec{\epsilon}_{i_k})$ for some infinite subsequence $\{x_{i_k}\}$. Hence there are $x'_{i_k}\in X_{i_k}$ such that $\vec{F}^{i_k}(x'_{i_k})>\vec{F}^{i_k}(x_{i_k})+\vec{\epsilon}_{i_k}$. Since $X$ is compact, without loss of generality, we can consider that $x'_{i_k}\rightarrow x'\in X$ and conclude
$$\vec{F}(x')\ge\limsup_k\vec{F}^{i_k}(x'_{i_k})\ge \lim_k\vec{F}^{i_k}(x_{i_k})+\lim_k\vec{\epsilon}_{i_k}= \vec{F}(x^*)+\vec{\epsilon}>\vec{F}(x^*)+\vec{\epsilon}\,', $$ 
that contradicts $\vec{\epsilon}\,'$-nondominance of $x^*$. The proof is complete. 

\section{Multi-criteria random search (MRS) algorithm and its convergence}

The next {\it random search algorithm} uses random discrete approximations $X_i=\cup_{k=1}^i \tilde{X}_k$  feasible set $X$ and estimates $\vec{F}^i(x),\;x\in X_i,$ of the objective function (\ref{MOP2}). The algorithm generates a random sequence of approximate solutions ${X}^*_i$, $i=1,2,...$, of the task (\ref{MOP2}) as follows. 

At the first iteration a first generation  of $N_1$ points $\tilde{X}_{1}$ is randomly generated in the set $X$, the estimates $\vec{F}^1(x)$ of the objective function $\vec{F}(x)$ are built for all points $x\in \tilde{X}_1$ and in the set $\{\vec{F}^1(x),\,x\in \tilde{X}_{1}\}$ a subset $\{\vec{F}^1(x),\,x\in {X}^*_{1}(\vec{\epsilon})\}$ of all $\vec{\epsilon}_1$-nondominated points is chosen. 

Suppose that at iteration $i$ we already have built the set ${X}^*_i$. Then (preferably in a vicinity of the set ${X}^*_{i}$) a new generation of $N_i$ random points $\tilde{X}_i$ is generated, estimates $\vec{F}^i(x)$ of the objective function $\vec{F}(x)$ for all $x\in X_i=\cup_{k=1}^i \tilde{X}_k$ are built, and from the set $\{\vec{F}^i(x),\,x\in X_i\}$ the subset $\{\vec{F}^i(x),\,x\in {X}^*_{i}\}$ of $\vec{\epsilon}_i$-nondominated points  is chosen, and then we proceed to iteration $i+1$. The process continues indefinitely long or ends at reaching the limit of iterations.

Below we formulate conditions of convergence of the MRS algorithm, announced in
\cite{NorkinBV_2015_DNASU}, \cite{NorkinBV_2015_KVT}. The next statement follows from Borel-Cantelli lemma.

\begin{lem}\label{lem2}
Let  a sequence of random sets $\{\tilde{X}_i, \;i=1,2,...\}$ be such that $\tilde{X}_i\subseteq X$ with probability one and with non-zero probability $p_i(x,\delta)>0$ the set $\tilde{X}_i$ intersects with any $\delta$-vicinity of any point $x\in X$, and it holds $\sum_i p_i(x,\delta)=+\infty$. Then with probability one $\limsup_i \tilde{X}_i=X$ and thus $\lim_i \cup_{k=1}^i \tilde{X}_k=X$. 
\end{lem}

{\it Proof.}
By assumption of the lemma, $\limsup_i \tilde{X}_i\subseteq X$. Let us show that with probability one the reverse inclusion is satisfied,  $X\subseteq\limsup_i \tilde{X}_i$. Let us choose a countable everywhere dense subset $X'$ in $X$ and show that with probability one  $\limsup_i \tilde{X}_i\supseteq X'$, from which the required assertion follows.  Let us fix an arbitrary point $x'\in X'$. By the conditions of the lemma, with probability one the sequence $\tilde{X}_i$ hits any $\delta$-vicinity of $x'$ the infinite number of times, so with probability one, there exists a subsequence $\{\tilde{X}_{i_k}\ni x_{i_k}\rightarrow x'\}$, i.e., with probability one $x'\in \limsup_i \tilde{X}_i$. The countability of $X'$ guarantees that with probability one $X'\subseteq \limsup_i \tilde{X}_i$, and by virtue of the density of $X'$ in $X$ and closeness of the set $\limsup_i \tilde{X}_i$ it follows that $X\subseteq \limsup_i \tilde{X}_i$ with probability one. The proof is complete. 

\begin{lem}\label{lem3}
Let $\tilde{X}_i\subseteq X$, $\limsup_i \tilde{X}_i=X,$ $\lim_i\vec{\epsilon}_i=\vec{\epsilon}$, 
$\vec{F}(x)$ is continuous on $X$ and $\Delta_i=\sup_{x\in X_i}\|\vec{F}^i(x)-\vec{F}(x)\|\rightarrow 0$. Then, for each $\vec{\epsilon}\,'<\vec{\epsilon}$, we have \[ X^*(\vec{\epsilon}\,')\subseteq\liminf_i \hat{X}_i(\vec{\epsilon})\subseteq \limsup_i \hat{X}_i(\vec{\epsilon})\subseteq X^*(\vec{\epsilon}). \] 
\end{lem}

{\it Proof.}
Note, that for $X_i=\cup_{k=1}^i \tilde{X}_k$ we get $\lim_i X_i=X$, so the assertion of the lemma  follows from Theorem \ref{theo1}. 

The above lemma actually covers the case of sample average approximation of the vector objective function outlined in example {\bf E5}.

As a consequence of Lemmas \ref{lem2}, \ref{lem3} we obtain the following result on the convergence of random search algorithm $\vec{\epsilon}$-nondominated set of the problem (\ref{MOP2}).

\begin{thm}
(Convergence of the MRS algorithm). Let vector function $\vec{F}(x)$ be continuous on a compact set $X\subset{\rm R}^m$, random sets $\tilde{X}_i$ with positive probability $p_i(x,\delta)>0$ intersect with any $\delta$-vicinity of each point $x\in X$, and $\sum_i p_i(x,\delta)=+\infty$. Denote $X_i=\cup_{k=1}^i\tilde{X}_k$. Let $\{\vec{F}^i(x),\,x\in X_i\subset X\}$ be a sequence of random vector functions such that with probability one $\Delta_i=\sup_{x\in X_i}\|\vec{F}^i(x)-\vec{F}(x)\|\rightarrow 0$.  Then, with probability one a) all cluster points of  $\{{X}^*_i(\vec{\epsilon}_i)\}$ belong to $X^*(\vec{\epsilon})$ and b) for each point $x^*\in X^*(\vec{\epsilon}\,')$, $\vec{\epsilon}\,'<\epsilon$, there is a sequence of points $\{x_i\in {X}^*_i(\vec{\epsilon}_i)\}$ convergent to $x^*$. 
\end{thm}

In particular, if $\vec{\epsilon}>0$, then for each weakly Pareto optimal point $x^*\in X^*\subseteq X$ there is a sequence of points $\{x_i\in {X}_i(\epsilon)\}$, convergent to $x^*$. And by virtue of upper semicontinuity of mapping $X^*(\vec{\epsilon})$ in $\vec{\epsilon}=\vec{0}$ for a sufficiently small vector $\vec{\epsilon}$ the set $X^*(\vec{\epsilon})$ will appear in an arbitrarily small neighborhood of the weakly Pareto optimal points $X^*=X^*(\vec{0})$.

\section {Multiobjective stochastic optimal control of insurance business in discrete time.} 

Our methodology of multiobjective stochastic optimization to solution of multiobjective actuarial stochastic optimal control problems is reviewed in \cite{Ermoliev_Norkins_2020}, \cite{NorkinBV_2024}.

Consider a controlled vector stochastic process of the following form \cite {Gihman_Skorohod_1977}: \begin{eqnarray}  
y^{t + 1}& =& f_ {t} (y^{0}, y^{1}, ..., y^{t}; \xi^{0}, \xi^{1}. .., \xi^{t}; u_ {t} (y^{0}, ..., y^{t}; \xi^{0}, ..., \xi^{t})), \nonumber\\
&& \;\;y^{0} = x_ {0}, \;\;t = 0,1, ..., T, \label{ZEqnNum961840}
\end{eqnarray} 
where $ t = 0,1, ..., T $ denotes the discrete time; $ \left \{y^{0}, y^{1}, ... \right \} $ is a sequence of states of the process; $ \left \{\xi^{0}, \xi^{1}, ... \right \} $ is an uncontrolled sequence of random variables that affect the state of the process; $ \{u_ {t} (\cdot)\in U_ {t}\} $ is a sequence of random controls selected from sets of admissible controls $ U_ {t} $; $ y^{0} = x_ {0} $ is an initial state of the process; $ \left \{f_ {t} (\cdot) \right \} $ is the process model. 

As an example we can consider the simplified case, in which the evolution of a capital $ y^t $ 
of an insurance firm
in a discrete time $ t = 0,1, ... $ can be described by equation \cite {Schmidli_2008}: \begin {equation} \label {risk_process} y^{t + 1} = \left \{\begin {array} {lc} {y^{t} -u_t + Z^t,} & {y^{t} \ge 0} \\{y^{t},} & {y^{t} <0,} \end {array} \right. , \end {equation} where $ y^0 $ is a seed capital, $ u_t $ is dividend amount, $ Z^t $ is a random insurance premium or claim at time $t$. 

Suppose that at each step $ t $ of process  (\ref {ZEqnNum961840}) the decision maker (DM), evaluates the process by means of functions $ r_ {ti} (y^{0}, ..., y^{t}; u_ {0} (\cdot), ..., u_ {k} (\cdot); \xi^{0}, ..., \xi^{k}) $, $ i = 1, ..., m $, and the total discounted estimates of the process for the $ T + $ 1 periods of time are given by 
\begin {equation} \label {ZEqnNum636784} I_ {i} \left (u_ {0} (\cdot), u_ {1} (\cdot), ..., u_ {T} (\cdot) \right) = {\rm E} \sum _ {k = 0}^{T} \gamma^{k} r_ {ki} (y^{0}, ..., y^{k}; u_ {0} (\cdot), ..., u_ {k} (\cdot); \xi^{0}, ..., \xi^{k}), 
\end {equation} 
where $ \gamma \in (0,1] $ is discount factor, $ i = 1, ..., m $. By choosing different functions $ r_ {ki} (y^{0}, ..., y^{k}; u_ {0} (\cdot), ..., u_ {k} (\cdot); \xi^{0}, ..., \xi^{k}) $  evaluation of different aspects of the model is possible. If, for example, 
\[r_ {ki^{*}} (y^{0}, ..., y^{k}; u_ {0} (\cdot), ..., u_ {k} (\cdot); \xi^{0}, ..., \xi^{k}) = \left \{\begin {array} {cc} {1,} & {\exists y^{k} \in A, \, \, k \le t,} \\{0,} & {y^{k} \notin A \, \, \, \forall k \le t,} \end {array} \right. \;\;\;\gamma = 1, \] 
then the corresponding indicator $ I_ {i^{*}} (\cdot) $ represents the probability of  the process hitting $ A. $ 

Thus, the problem of stochastic multi-criteria (Pareto-optimal) control has the form 
\begin {equation} 
\label {ZEqnNum508990} \left [I_ {i} \left (u_ {0} (\cdot), u_ {1} (\cdot), ..., u_ {T} (\cdot) \right), \, \, i = 1, ..., m \right] \to {\rm extr} _ {\left \{u_ {t} (\cdot) \in U_ {t}, \, \, t = 0, ..., T \right \}}. 
\end {equation} 

The essential difficulty in solving the problem (\ref {ZEqnNum508990}) is the calculation or estimation of the mathematical expectation (\ref {ZEqnNum636784}) over all possible trajectories of the process (\ref {ZEqnNum961840}). In general, this can only be done by Monte Carlo method. Another difficulty is that the problem (\ref {ZEqnNum508990}) is infinite-dimensional. 

Problem (\ref {ZEqnNum508990}) is considerably simplified if to restrict the set, where we search for optimal controls, to the class of parametrically defined functions $ U_ {t} = \left \{u_ {t} (y^{0}, ..., y^{t} ; \xi^{0}, ..., \xi^{t}; x^{t}), \, \, x^{t} \in X^{t} \right \} $, where $ x^{t} $ is a finite-dimensional parameter. Then the functional $ I_ {i} $ becomes a function of the finite-dimensional parameters, and the optimization problem (\ref {ZEqnNum508990}) is transformed into a finite-dimensional vector stochastic programming problem.

Another way of simplification is to consider (vector) Markov processes: 
\begin {equation} 
\label {ZEqnNum602415} y ^ {t + 1} = f (y ^ {t}; \xi ^ {t}; u (y ^ {t}; \xi ^ {t})),\;\; y ^ {0} = x_ {0}, \;\;t = 0,1, ..., T, 
\end {equation} 
where the model $ f $ and control $ u $ are not changed with time, and the next state $ y ^ {t + 1} $ depends only on the current state $ y ^ {t} $ of the system  and on the current state of the environment $ \xi ^ {t} $. In many cases, the optimal control $ u_ {t} (\cdot) $ can be found in the class of functions that depend only on $ (y ^ {t}; \xi ^ {t}) $ or $ y ^ {t} $ \cite {Gihman_Skorohod_1977}. In particular, optimal control of dividends in dynamic models of insurance companies   (\ref {risk_process}) often takes on the form of the so-called barrier strategy \cite {Schmidli_2008}, \cite {Albrecher_Thonhauser_2009}, \cite {Avanzi_2009},  where dividends are paid only if the company's capital $ y ^ t $ exceeds a certain threshold $ b $ (barrier). If the functional form of the control is chosen and depends only on $ n $-dimensional parameter $ x \in {\rm R} ^ {n} $, i.e. $ u (y, \xi, x) $, then (\ref {ZEqnNum508990}) again becomes the problem of a finite dimensional  multi-criteria stochastic programming, 
\begin {equation} 
\label {ZEqnNum247227} \left [I_ {i} ^ {T} \left (x_ {0}, x \right) = {\rm E} \sum _ {k = 0} ^ {T } \gamma ^ {k} r_ {ki} (y ^ {k}, \xi ^ {k}, u (y ^ {k}, \xi ^ {k}, x)), \, \, i = 1, ..., m \right] \to {\rm extr} _ {x \in X}. 
\end {equation}

A separate problem is the computation of indicator values $ I_ {i} \left (x_ {0}, x \right) $ for some fixed parameter vector $ \left (x_ {0}, x \right) $, because they are mathematical expectations over random trajectories of process (\ref {ZEqnNum602415}). Even if the random variables $ \xi ^ {k} $ are discrete with known distributions, finding the expectation in (\ref {ZEqnNum247227}) requires summation over all possible paths of the process (\ref {ZEqnNum602415}), which appears to be problematic. A universal method for estimating integrals (\ref {ZEqnNum636784}) is the Monte Carlo method, which, however, may require a very large number of trials to achieve acceptable accuracy. Note that simulating trajectories of the process (\ref {ZEqnNum961840}) can be performed in {\it parallel} that significantly reduces the computation time of the Monte Carlo method. An alternative approach of finding indicators $ I_ {i} ^ {T} \left (x_ {0}, x \right) $ is solving  Bellman type integral equations  \cite {Norkin_2014_CSA5}. It appears that under certain conditions on Markov process (\ref {ZEqnNum602415}), indicators $ I_ {i} ^ {t} \left (x_ {0}, x \right) $, $ t = 0,1, ... , T, $ satisfy relations \cite {Norkin_2014_CSA5}:
\[I_{i}^{t} \left(x_{0} ,x\right)={\rm E}r_{(T-t)i} (x_{0} ,\xi ,u(x_{0} ,\xi ,x))+\gamma {\rm E}I_{i}^{t-1} \left(f(x_{0} ;\xi ;u(x_{0} ,\xi ,x)),x\right), \] 
\begin{equation} \label{ZEqnNum170085} I_{i}^{0} \left(x_{0} ,x\right)={\rm E}r_{Ti} (x_{0} ,\xi ,u(x_{0} ,\xi ,x)),  \;\;\;t=1,...,T.      \end{equation} 

If the random variable $ \xi \in \left \{\xi _ {1}, ..., \xi _ {S} \right \} $ is discrete with known probabilities $ p_ {s} $ of its realizations $ \xi _ {s} $, then (\ref {ZEqnNum170085}) are turn into deterministic relations,
\[I_{i}^{t} \left(x_{0} ,x\right)=\sum _{s=1}^{S}
r_{(T-t)i} (x_{0} ,\xi _{s} ,u(x_{0} ,\xi _{s} ,x))p_s +
\gamma \sum _{s=0}^{S}
I_{i}^{t-1} \left(f(x_{0} ;\xi _{s} ;u(x_{0} ,\xi _{s} ,x)),x\right)p_s ,  \] 
\[I_{i}^{0} \left(x_{0} ,x\right)=\sum _{s=1}^{S}
r_{Ti} (x_{0} ,\xi _{s} ,u(x_{0} ,\xi _{s} ,x))p_s ,  \;\;\;t=0,...,T-1,\] 
which can be used to calculate $ I_ {i} ^ {T} \left (x_ {0}, x \right) $. To do this, we have to define a grid over a one-dimensional variable $ x_ {0} $ and successively compute $ I_ {i} ^ {0} \left (\cdot, x \right) $, $ I_ {i} ^ {1} \left (\cdot, x \right) $, \dots, $ I_ {i} ^ {T} \left (\cdot, x \right) $ in the grid nodes, and do interpolation in between. Note that this iterative process possesses natural {\it} parallelism property. The values of $ I_ {i} ^ {t} \left (\cdot, x \right) $ in the various grid points are calculated independently of each other on the basis of the discrete approximation of the function $ I_ {i} ^ {t-1} \left (\cdot, x \right) $, obtained in the previous iteration. Therefore, these calculations are similar to \cite [equation (7)] {Norkin_2011_HPC-UA} and can be {\it parallelized}.

The presented methodology of solution of multiobjective stochastic optimization problems were
implemented  in decision support systems for optimization of insurance business \cite {Norkin_2012_KPIeng}, \cite {Norkin_2014_CSA2}, \cite{NorkinBV_2014_TAAC}.

\section {Conclusions} 

The article describes an information technology of multi-criterion stochastic optimization of "input- random output" systems. In practice, these models are highly nonlinear and non-convex. Their functioning can generally be described by deterministic vector parameters such as means, quantiles, probabilities of reaching/exiting specified areas, etc. Collapsing the vector component into a scalar indicator with the view of its optimization is not always possible. So, the task is to find such inputs, that correspond to the Pareto-optimal performance indicator vectors. This paper proposes a methodology of solve such problem by visualized interactive {\it parallel} random search with selection of Pareto-optimal points (optimization cloud). The technology is illustrated by an insurance multi-criteria optimization support system.

\bibliographystyle{unsrt}

\end{document}